\documentclass[12pt]{article}
\usepackage{amsmath,amsthm,amsfonts,amssymb}
\newtheorem{thm}{Theorem}[section]
\newtheorem{theorem}[thm]{Theorem}
\newtheorem{cor}[thm]{Corollary}

\newtheorem{lem}[thm]{Lemma} 
\newtheorem{ques}[thm]{Question} 
\newtheorem{prop}[thm]{Proposition}
 
\newtheorem{conj}[thm]{Conjecture} 
\newtheorem{defn}[thm]{Definition} 
\newtheorem{ex}[thm]{Example}

\def\D{{\cal D}}
\newcommand{\R}{{\rm R}}
\newcommand{\Z}{{\rm Z}}

\newcommand{\M}{{\cal M}}
\newcommand{\Q}{{\rm Q}}
\newcommand{\T}{{\rm T}}
\newcommand{\N}{\mathcal N}
\newcommand{\G}{\mathcal G}
\newcommand{\HH}{\mathcal H}

\newcommand{\X}{\mathfrak{X}}
\DeclareMathOperator{\supp}{supp}
\DeclareMathOperator{\egr}{egr}

\newcommand{\Fix}{{\rm Fix}}

\newcommand{\Per}{{\rm Per}}
\newcommand{\Homeo}{{\rm Homeo}}

\newcommand{\Diff}{{\rm Diff}}
\newcommand{\bl}{black}
\newcommand{\RP}{{\mathbb R \mathbb P}}
\def\ti{\tilde}
\def\sinfty{S_{\infty}}

\usepackage[usenames]{color}
\usepackage{latexsym}
\definecolor {darkgreen}{rgb}{0,0.6,0}
\title{\color{\bl}Distortion in Groups of Circle and Surface Diffeomorphisms}
\author{John Franks\thanks{Supported in part by NSF grant DMS0099640.}}
\date{\today}

\begin{document}
\maketitle
\section{Introduction}
In his seminal article \cite{Sm} S.~Smale outlined a program for 
the investigation of the properties of generic smooth dynamical systems.
He proposed as definition of the object of study the smooth
action of a non-compact Lie group $\G$ on a manifold $M$; i.e.,
a smooth function 
\[
f: \G \times M \to M
\]
satisfying $f(g_1, f(g_2, x)) = f(g_1g_2, x)$ 
and $f(e, x) = x$ for all
$x \in M$ and all $g_1, g_2 \in \G$, where $e$ is the identity
of $\G$.  Equivalently one can consider the homomorphism
\[
\phi: \G \to \Diff(M)
\]
from $\G$ to the group of diffeomorphisms of $M$ 
given by $\phi(g)(x) = f(g,x).$
The primary motivation, and by far the most studied case, has been
that where $\G$ is either the Lie group $\R$ of real numbers or
the discrete group $\Z$. As noted in the Introduction to this volume
this study grew out of an interest in solution of differential
equations where the group $\R$ or $\Z$ represents time 
(continuous or discrete).  

In this article we will focus on the far less investigated case where
$\G$ is a subgroup of Lie group of dimension greater than one.  The
continuous and discrete cases when $\G$ is $\R$ or $\Z$ share many
characteristics with each other and it is often clear how to formulate
(or even prove) an analogous result in one context based on a result
in the other.  Very similar techniques can be used in the two
contexts. However, when we move to more complicated groups the
difference between the actions of a connected Lie group and the
actions of a discrete subgroup become much more pronounced.  One must
start with new techniques in the investigation of actions of a
discrete subgroup of a Lie group.

As in the case of actions by $\R$ and $\Z$ one can impose additional
structures on $M$, such as a volume form or symplectic form, and
require that the group $\G$ preserve them.  For this article we
consider manifolds of dimension two where the notion of volume
form and symplectic form coincide.  As it happens many of the
results we will discuss are valid when a weaker structure, namely
a Borel probability measure, is preserved.

The main object of this article is to provide some context for, and an
exposition of, joint work of the author and Michael Handel which can
be found in \cite{FH3}.

The ultimate aim is the study of the (non)-existence of actions of
lattices in a large class of non-compact Lie groups on surfaces.  A
definitive analysis of the analogous question for actions on $S^1$ was
carried out by \'E.~Ghys in \cite{G}.  Our approach is topological and
insofar as possible we try to isolate properties of a group which
provide the tools necessary for our analysis.  The two key properties
we consider are almost simplicity of a group and the existence of a
distortion element.  Both are defined and described below.

We will be discussing groups of homeomorphisms and diffeomorphisms
of the circle $S^1$ and of a compact surface $S$ without boundary.  
We will denote the group of $C^1$ diffeomorphisms which preserve
orientation by $\Diff(X)$ where $X$ is $S^1$ or $S$.  Orientation
preserving homeomorphisms will be denoted by $\Homeo(X)$.  If $\mu$
is a Borel probability measure on $X$ then $\Diff_\mu(X)$
and $\Homeo_\mu(X)$ will denote the respective subgroups which preserve
$\mu.$ Finally for a surface $S$ we will denote by
$\Diff_\mu(S)_0$ the subgroup of $\Diff_\mu(S)$ of elements isotopic
to the identity.

An important motivating conjecture is the following.

\begin{conj}[R. Zimmer \cite{Z2}]
Any $C^\infty$ volume preserving action of  $SL(n,\Z)$ 
on a compact manifold with dimension less than $n$, factors
through an action of a finite group.
\end{conj}

This conjecture suggests a kind of exceptional rigidity of 
actions of  $SL(n,\Z)$ on manifolds of dimension less than
$n$.  The following result of D. Witte, which is a special case
of his results in \cite{W}, shows that in the case of $n=3$ and
actions on $S^1$ there is indeed a very strong rigidity.

\begin{theorem}[D. Witte \cite{W}]\label{thm:Witte}
Let $\G$ be a finite index subgroup of $SL(n,\Z)$ with
$n \ge 3.$
Any homomorphism 
\[
\phi: \G \to \Homeo(S^1)
\] has a finite
image.
\end{theorem}

\begin{proof}
We first consider the case $n=3.$
If $G$ has finite index in $SL(3, \Z)$ then there is $k > 0$ such that
\[
a_1 = \begin{pmatrix}
1 & k & 0\\
0 & 1 & 0\\
0 & 0 & 1\\
\end{pmatrix},
a_2 =
\begin{pmatrix}
1 & 0 & k\\
0 & 1 & 0\\
0 & 0 & 1\\
\end{pmatrix},
a_3 =
\begin{pmatrix}
1 & 0 & 0\\
0 & 1 & k\\
0 & 0 & 1\\
\end{pmatrix},
\]
\[
a_4 = \begin{pmatrix}
1 & 0 & 0\\
k & 1 & 0\\
0 & 0 & 1\\
\end{pmatrix},
a_5 = \begin{pmatrix}
1 & 0 & 0\\
0 & 1 & 0\\
k & 0 & 1\\
\end{pmatrix},
\text{ and }
a_6 = \begin{pmatrix}
1 & 0 & 0\\
0 & 1 & 0\\
0 & k & 1\\
\end{pmatrix},
\]
are all in $\G.$  We will show that each of the $a_i^k$ is in the kernel of $\phi.$
A result of Margulis (see Theorem~\ref{thm:margulis} below) then implies that the
kernel of $\phi$ has finite index.  This result also implies that the case
$n=3$ is sufficient to prove the general result.

A straightforward computation shows that
$[a_i, a_{i+1}] = e$ and $[a_{i-1}, a_{i+1}] = a_i^{\pm k},$ where the 
subscripts are taken modulo $6$.
Indeed $[a_{i-1}^m, a_{i+1}^n] = a_i^{\pm mnk}.$

Let $g_i = \phi(a_i)$.  The group $\HH$ generated by $g_1$ and $g_3$
is nilpotent and contains $g_2^k$ in its center.  Since nilpotent
groups are amenable there is an invariant measure for the group $\HH$
and hence the rotation number $\rho : \HH \to \R/\Z$ is a
homomorphism.  Since $g_2^k$ is a commutator, it follows that $g_2^k$
has zero rotation number and hence it has a fixed point. A similar
argument shows that for all $i,\ g_i^k$ has a fixed point.

We will assume that one of the $g_i^k,$ for definiteness say
$g_1^k,$ is not the identity and show this leads to a contradiction.

Let $U_1$ be any component of $S^1 \setminus \Fix(g_1^k)$.  Then we
claim that there is a $U_2 \subset S^1$ which properly contains $U_1$
and such that $U_2$ is either a component of $S^1 \setminus
\Fix(g_{6}^k)$ or a component of $S^1 \setminus \Fix(g_{2}^k)$. We
postpone the proof of the claim and complete the proof.

Assuming the claim suppose that $U_2$ is a component of $S^1 \setminus
\Fix(g_2^k)$ the other case being similar.  Then again applying the
claim, this time to $g_2^k$ we see there is $U_3$ which properly
contains $U_2$ and must a component of $S^1 \setminus \Fix(g_3^k)$
since otherwise $U_1$ would properly contain itself.  But repeating
this we obtain proper inclusions
\[
U_1 \subset U_2 \dots U_5 \subset U_6 \subset U_1,
\]
which is a contradiction.  Hence $g_1^k = id$ which implies that
$a_1^k \in Ker(\phi).$  A further application of the result of 
Margulis (Theorem~\ref{thm:margulis} below) implies that 
$Ker(\phi)$ has finite index in $\G$ and hence that $\phi(\G)$ is
finite.

To prove the claim we note that $U_1$ is an interval whose endpoints
are fixed by $g_1^k$  and we will will first prove that it is
impossible for these endpoints also to be fixed by $g_6^k$ and
$g_2^k$.  This is because in this case we consider the action induced
by the two  homeomorphisms $\{g_{6}^k, g_{2}^k\}$
on the circle obtained by quotienting $U_1$ by $g_1^k$.  These two circle
homeomorphisms commute because $[g_{6}^k, g_{2}^k] = g_1^{\pm k^2}$ on
$\R$ so passing to the quotient where $g_1$ acts as the identity
we obtain a trivial commutator.
It is an easy exercise to see that if two degree one  homeomorphisms
of the circle, $f$ and $g$,  commute then any two lifts
to the universal cover  must also commute. (E.g. show that
$[\ti f, \ti g]^n$ is uniformly bounded independent of $n$.)
But this is impossible in our case because  the universal cover
is just $U_1$ and  $[g_{6}^k, g_{2}^k] = g_1^{\pm k^2} \ne id.$

To finish the proof of the claim we note that if $U_1$ contains a point 
$b \in \Fix( g_2^k)$ then $g_1^{nk}(b) \in \Fix( g_2^k)$ for all $n$ and hence
\[
\lim_{n \to \infty}g_i^{nk}(b)\text{ and } \lim_{n \to -\infty}g_i^{nk}(b),
\]
which are the two endpoints of $U_1$ must be fixed by $g_2^k$.  A similar argument
applies to $g_6^k.$  

It follows that at least one of $g_6^k$  and $g_2^k$  has no fixed points in $U_1$ and
does not fix both endpoints.  I.e. there is $U_2$ as claimed.
\end{proof}

It is natural to ask the analogous question for surfaces.

\begin{ex}
The group $SL(3,\Z)$ acts smoothly on $S^2$ by projectivizing the standard
action on $\R^3.$
\end{ex}

Consider $S^2$ as the set of unit vectors in $\R^3.$ If $x \in S^2$ and
$g \in SL(3,\Z),$
we can define $\phi(g) : S^2 \to S^2$ by
\[
\phi(g)(x) = \frac{gx}{|gx|}.
\]

\begin{ques}
Can the group $SL(3,\Z)$ act continuously or smoothly on a surface of genus at least one?
Can the group $SL(4,\Z)$ act continuously or smoothly on $S^2$?
\end{ques}

\section{Distortion in Groups}

A key concept in our analysis of groups of surface homeomorphisms
is the following.

\begin{defn}
An element
$g$ in a finitely generated group $G$ is called
{\color{\bl}\em distorted} if it has infinite order and
\[
{
\liminf_{n \to \infty} \frac{|g^n|}{n} = 0,
}
\]
where $|g|$ denotes the minimal word length of $g$ in some set of
generators.  If $\G$ is not finitely generated then
$g$ is distorted if it is distorted in some
finitely generated subgroup.
\end{defn}

It is not difficult to show that if $\G$ is finitely generated then the property of being
a distortion element is independent of the choice of generating set.

\begin{ex}
The subgroup $G$ of $SL(2,\R)$ generated by 
\[
A =
\begin{pmatrix}
1/2 &  0\\
0 & 2 \\
\end{pmatrix}
\text{ and }
B =
\begin{pmatrix}
1 &  1\\
0 & 1 \\
\end{pmatrix}
\]
satisfies 
\[
A^{-1}BA =
\begin{pmatrix}
1 &  4\\
0 & 1 \\
\end{pmatrix}
= B^4 \text{ and }
A^{-n}BA^n = B^{4^n}
\]
so $B$ is distorted.
\end{ex}

\begin{ex}
The group of integer matrices of the form
\[
\begin{pmatrix}
1 & a & b\\
0 & 1 & c\\
0 & 0 & 1\\
\end{pmatrix}
\]
is called the {\color{\bl}\em Heisenberg group}.
\end{ex}

If
\[
g =
\begin{pmatrix}
1 & 1 & 0\\
0 & 1 & 0\\
0 & 0 & 1\\
\end{pmatrix}
\text{ and }
h =
\begin{pmatrix}
1 & 0 & 0\\
0 & 1 & 1\\
0 & 0 & 1\\
\end{pmatrix}
\]
then their {\em commutator} $f = [g,h] := g^{-1}h^{-1}gh$ is
\[
f = 
\begin{pmatrix}
1 & 0 & 1\\
0 & 1 & 0\\
0 & 0 & 1\\
\end{pmatrix}
\text{ and }
{\color{\bl}f \text{ commutes with } g \text{ and } h.}
\]
This implies
\[
{\color{\bl}[g^n, h^n] = f^{n^2}}
\]
 so $f$ is distorted.

Let $\omega$ denote Lebesgue measure on the torus $\T^2.$.

\begin{ex}[G. Mess \cite{M}]
In the subgroup of $\Diff_\omega(\T^2)$ generated by the automorphism
given by 
\[
A = \begin{pmatrix}
2 & 1\\
1 & 1\\
\end{pmatrix}
\]
and a translation $T(x) = x + w$ where $w \ne 0$ is parallel to the 
unstable manifold of $A$, the element $T$ is distorted.
\end{ex}

\begin{proof}
Let $\lambda$ be the expanding eigenvalue of $A$.
The element 
$h_n = A^n T A^{-n}$ satisfies $h_n(x) = x + \lambda^n w$
and $g_n = A^{-n} T A^n$ satisfies $g_n(x) = x + \lambda^{-n} w$.
Hence $g_n h_n(x) = x +  (\lambda^n + \lambda^{-n}) w.$
Since $tr A^n = \lambda^n + \lambda^{-n}$ is an integer we conclude
\[
T^{tr A^n} = g_n h_n, \text{ so } |T^{tr A^n}| \le  4n +2.
\]
But
\[
\lim_{n \to \infty} \frac{n}{tr A^n} = 0,
\]
so $T$ is distorted.
\end{proof}

\begin{ques}
Is an irrational rotation of $S^1$ distorted in $\Diff(S^1)$ or
$\Homeo(S^1)?$  Is an irrational rotation of $S^2$ distorted in
$\Diff(S^2)$ or in the group of area preserving diffeomorphisms of $S^2?$ 
\end{ques}

\begin{ex}[D. Calegari \cite{C}]
There is a $C^0$ action of the Heisenberg group on $S^2$ 
whose center is generated by an irrational
rotation.  Hence an irrational rotation of $S^2$ is distorted
in $\Homeo(S^2).$
\end{ex}

\begin{proof}
Consider the homeomorphisms of $\R^2$ given by 
\[
G = \begin{pmatrix}
1 & 1\\
0 & 1\\
\end{pmatrix}
\]
and a translation $H(x,y) = (x ,y +1)$.  We compute
$F = [G,H]$ to be a translation $F(x,y) = (x+1, y).$
This defines an action of the Heisenberg group on $\R^2$.
Let $C$ be the cylinder obtained by quotienting by the 
relation $(x ,y ) \sim (x + \alpha, y)$ for some 
$\alpha \in \R \setminus \Q$. The quotient
action is well defined.  The two ends of $C$ are fixed by every element
of the action and hence if we compactify $C$ to obtain
$S^2$ by adding a point at each end, we obtain an action
of the Heisenberg group on $S^2.$ 
\end{proof}

A theorem of Lubotzky, Mozes, and Raghunathan shows that there
is a large class of non-uniform lattices which contain a distortion
element.

\color{\bl}\begin{theorem}[Lubotzky-Mozes-Raghunathan \cite{lmr}]\label{thm:lmr}
Suppose $\Gamma$ is a non-uniform irreducible lattice in a semi-simple 
Lie group $\G$ with $\R-$rank $\ge 2.$  Suppose further that $\G$ is
connected, with finite center and no nontrivial compact factors.
Then $\Gamma$ has distortion elements, in fact, elements whose 
word length growth is at most logarithmic.
\end{theorem}

\section{Distortion in almost simple groups}

\begin{defn}
A group is called {\color{\bl}\em almost simple} if every normal subgroup is
finite or has finite index.
\end{defn}

As we saw in the proof of the theorem of Witte (Theorem \ref{thm:Witte}),
the fact that $SL(n, \Z)$ is almost simple when $n \ge 3$ plays a crucial
role.  This will also be true for our analysis of surface diffeomorphisms.

\begin{theorem}[Margulis \cite{Mar}]\label{thm:margulis}
Assume $\Gamma$ is an irreducible lattice in a semi-simple Lie group with
$\R-$rank $\ge 2,$ e.g. any finite index subgroup of $SL(n, \Z)$ 
with $n \ge 3$.
Then $\Gamma$ is almost simple.
\end{theorem}

The following observation is a very easy consequence of the fact
that $\R$ has no distortion elements and no elements of finite
order. Nevertheless, it is a powerful tool in our investigations.

\begin{prop}[\cite{FH3}]\label{prop}
If $\G$ is a finitely generated almost simple group which
contains a distortion element and $\HH \subset \G$ is a
normal subgroup, then the only homomorphism from $\HH$
to $\R$ is the trivial one.
\end{prop}

\begin{proof}
Since $\G$ is almost simple, $\HH$ is either finite or has finite index.
Clearly the result is true if $\HH$ is finite, so we assume it has
finite index. If $u$ is a distortion element in $\G$ then $v:=u^k \in
\HH$ for some $k > 0$.  Let $\D$ be the smallest normal subgroup 
of $\G$ containing $v$, i.e. the group generated by 
$\{ g^{-1} v g\ |\ g \in \G\}.$  Then $\D$ is infinite and normal
and hence has finite index in $\G$; it is obviously contained in $\HH$.  Thus
$\D$ has finite index in $\HH$.  Since $\R$ contains neither torsion
nor distortion elements, $v$, and hence $\D$ is in the kernel of
$\psi$ for every homomorphism $\psi: \HH \to \R$.  Since $\D$ has
finite index in $\HH$ we conclude that $\psi(\HH)$ is finite and hence
trivial.

\end{proof}

The last important ingredient we will need is the following result
of Thurston, originally motivated by the study of foliations.

\begin{theorem}[Thurston stability theorem \cite{Th}]
Let $\G$ be a finitely generated group and $M$ a connected
manifold.  Suppose
\[
\phi: \G \to \Diff^1(M)
\]
is a homomorphism and there is
$x_0 \in M$ such that for all $g \in \phi(\G)$
\[
g(x_0) = x_0 \text{ and } Dg(x_0) = I.
\]
Then either $\phi$ is trivial or there is a non-trivial
homomorphism from $\G$ to $\R$.
\end{theorem}

\begin{proof}
The proof we give is due to W.~Schachermayer \cite{S}.
Let $\{g_i\}$ be a set of generators for $\phi(\G).$
The proof is local so there is no loss of generality
in  assuming $M = \R^m$ and that $x_0 = 0$ is 
not in the interior of the points fixed by all of $\phi(\G).$

For $g \in \phi(\G)$ let 
$\widehat g(x) = g(x) - x,$ so $g(x) = x + \widehat g(x)$ and $D\widehat g(0) = 0.$
We compute 
\begin{align*}\label{eqn1}
\widehat{gh}(x) &= g(h(x)) -x \\
&= h(x) - x + g(h(x)) - h(x)\\
&= \widehat h(x) + \widehat g(h(x))\\
&= \widehat h(x) + \widehat g( x + \widehat h(x))\\
&= \widehat g(x) + \widehat h(x) + \big (\widehat g( x + \widehat h(x)) - \widehat g(x)\big ).
\end{align*}

Hence we have shown that for all $g,h \in \G$ and for all $x \in \R^m$
\begin{equation}\label{eqn1}
\widehat{gh}(x)
= \widehat g(x) + \widehat h(x) + \big (\widehat g( x + \widehat h(x)) - \widehat g(x)\big ).
\end{equation}

Choose a sequence $\{x_n\}$ in $\R^m$ converging to $0$ such that for
some $i$ we have $|\widehat g_i(x_n)| \ne 0$ for all $n$. This is
possible since $0$ is not in the interior of the points fixed by all
of $\phi(\G).$

Let $M_n = \max \{|\widehat g_1( x_n)|, \dots, |\widehat g_k( x_n)|\}.$   Passing
to a subsequence we may assume that for each $i$ the limit
\[
L_i = \lim_{n \to \infty} \frac{\widehat g_i( x_n)}{M_n} 
\]
exists and that $\| L_i\| \le 1.$  For some $i$ we have
$\|L_i\| = 1$; for definiteness say for $i = 1$.

If $g$ is an arbitrary element of $\G$ such that the limit
\[
L = \lim_{n \to \infty} \frac{\widehat g( x_n)}{M_n} 
\]
exists then for each $i$ we will show that
\[
\lim_{n \to \infty} \frac{\widehat {g_i g}( x_n)}{M_n}  = L_i + L.
\]
Indeed because of Equation~(\ref{eqn1}) above it suffices to show
\begin{equation}\label{eqn2}
\lim_{n \to \infty} \frac{\widehat g_i( x_n + \widehat g(x_n)) - \widehat g_i(x_n)))}{M_n}  = 0.
\end{equation}
By the mean value theorem 
\[
\lim_{n \to \infty} \Big \| \frac{\widehat g_i( x_n + \widehat g(x_n)) - \widehat g_i(x_n)))}{M_n}\Big \|
\le 
\lim_{n \to \infty} \sup_{t \in [0,1]}  \|D\widehat g_i(z_n(t))\|
\Big \| \frac{ \widehat {g}( x_n)}{M_n}\Big \|,
\]
where $z_n(t) = x_n + t \widehat g(x_n).$
But
\[
\lim_{n \to \infty} \frac{\widehat g( x_n)}{M_n} = L \text{ and }
\lim_{n \to \infty} sup_{t \in [0,1]}  \|D\widehat g_i(z_n(t))\| = 0,
\]
since $D\widehat g_i(0) = 0$ and hence Equation (\ref{eqn2}) is established.

It follows that if we define $\Theta :\phi(\G) \to \R^m$ by
\[
\Theta(g) = \lim_{n \to \infty} \frac{\widehat g( x_n)}{M_n} 
\]
this gives a well defined homomorphism from $\phi(\G)$ to $\R^m.$
\end{proof}

The following theorem is much weaker than known results on this topic,
for example the theorem of Witte cited above or the definitive results
of \'E.~Ghys \cite{G} on $C^1$ actions of lattices on $S^1$.  For
those interested in circle actions the articles of Ghys, \cite{G} and
\cite{G2}, are recommended.  We present this ``toy'' theorem because
its proof is simple and this is the proof which we are able to
generalize to surfaces.

\begin{thm}[Toy Theorem]
Suppose $\G$ is a finitely generated almost simple group
and has a distortion element and suppose $\mu$ is a finite
probability measure on $S^1$.   If 
\[
\phi: \G \to \Diff_\mu( S^1) 
\]
is a homomorphism then $\phi(\G)$ is finite.
\end{thm}
\begin{proof}
We give a sketch of the proof.
The rotation number $\rho: \Diff_\mu( S^1) \to \R/\Z$
is a homomorphism because the group preserves an invariant
measure.
If $f$ is distorted then $\rho(f)$ has finite order in $\R/\Z$
since there are no distortion elements in $\R/\Z$.
Thus for some $n>0,\ \rho(f^n) = 0$ and $\Fix(f^n)$ is non-empty.

For any homeomorphism of $S^1$ leaving invariant a probability measure
$\mu$ and having fixed points the support $\supp(\mu )$ is a subset
of the fixed point set.  Hence  $\supp(\mu )\subset \Fix(f^n)$.

Define  $\G_0 := \{g \in \G\ |\ \phi(g) \text{ pointwise fixes }
\supp(\mu)\}.$ It is infinite, since $f^n \in \G_0$, and it is 
normal in $\G$. Hence it has finite index in $\G.$
It follows that $\phi(\G_0)$ is trivial.  This is because at
a point $x \in \supp(\mu)$ the homomorphism from $\G_0$ to
the multiplicative group  $\R^+$
given by $g \mapsto D\phi(g)_x$ must
be trivial by Proposition~\ref{prop} above.
Hence we may use the Thurston stability theorem (and another
application of Proposition~\ref{prop}) to conclude that
$\phi(\G_0)$ is trivial.  Since $\G_0$ has finite index in
$\G$ the result follows.
\end{proof}

We proceed now to indicate how the proof of the ``toy theorem'' generalizes
to the case of surfaces.
The statement that $\supp(\mu )\subset \Fix(f^n)$ if $\Fix(f^n)$ is non-empty,
is trivial for the circle, but generally false for surfaces.
Nevertheless, it was a key ingredient of the proof of the ``toy theorem.''
This apparent gap is filled by the following theorem from \cite{FH3}.

\begin{thm}[\cite{FH3}]\label{thm:distort}
Suppose that $S$ is a closed oriented surface, that $f$ is a distortion element in $\Diff(S)_0$ and that $\mu$ is an $f$-invariant Borel probability measure.
\begin{enumerate}
\item
If $S$ has genus at least two then  $\Per(f) = \Fix(f)$ and
$\supp(\mu) \subset \Fix(f)$.
\item
If $S = T^2$ and $\Per(f) \ne \emptyset$, then
all points of $\Per(f)$ have the same period, say $n$, 
and $\supp(\mu) \subset \Fix(f^n)$ 
\item
If $S = S^2$ and if $f^n$ has at least three fixed points
for some smallest $n>0$, then $\Per(f) = \Fix(f^n)$ and 
$\supp(\mu) \subset \Fix(f^n)$.
\end{enumerate}
\end{thm}

We can now nearly copy the proof of the ``Toy Theorem'' to 
obtain the following.

\begin{thm}[\cite{FH3}]\label{thm:lattice}
Suppose $S$ is a closed oriented surface of genus at least one and
$\mu$ is a Borel probability measure on $S$ with infinite support.
Suppose $\G$ is finitely generated, almost simple and has a 
distortion element.  Then any homomorphism
\[
\phi: \G \to \Diff_\mu(S)
\]
has finite image.
\end{thm}

\begin{proof}
We present only the case that $S$ has genus greater than one.
Define  $\G_0 := \{g \in \G\ |\ \phi(g) \text{ pointwise fixes }
\supp(\mu)\}.$ It is infinite, since by Theorem~\ref{thm:distort}
the distortion element is in $\G_0$, and it is 
normal in $\G$. 
Hence $\G_0$  has finite index in $\G.$

We wish to show that $\phi(\G_0)$ is trivial using the Thurston
stability theorem.  Let $x$ be a point in the frontier of $\supp(\mu)$
which is an accumulation point of $\supp(\mu)$.  There is then a unit
tangent vector $v \in TM_x$ which is fixed by $D\phi(g)_x$ for 
all $g \in \G_0$.  If we denote the unit sphere in the tangent
space $TM_x$ by $S^1$ then projectivization of $D\phi(g)_x$ 
gives an action of $\G_0$ on $S^1$ with global fixed point 
$v$.  There is then a homomorphism from $\G_0$ to $\R^+$ given 
by mapping $g$ to the derivative at $v$ of the action of $g$
on $S^1.$  This must be trivial by Proposition~\ref{prop} above.  
Hence we may apply the Thurston stability theorem to the
action of $\G_0$ on $S^1$ to conclude that it is trivial,
i.e., that $D\phi(g)_x = I$ for all $g \in \G_0$.
We may now apply the Thurston stability theorem to the action
of $\G_0$ on $S$ to conclude that
$\phi(\G_0)$ is trivial.  Since $\G_0$ has finite index in
$\G$ the result follows.
\end{proof}

This result was previously known
in the special case of symplectic diffeomorphisms by a result 
of L. Polterovich \cite{P}.

The result above also holds with $\supp(\mu)$ finite if 
$\G$ is a Kazhdan group (aka $\G$ has property T).
(see \cite{K})

The fact that the hypotheses of Theorem \ref{thm:lattice} are
satisfied by a large class of non-uniform lattices
follows from the result of Lubotzky, Mozes, and Raghunathan,
Theorem \ref{thm:lmr}, together with
Theorem \ref{thm:margulis},
the Margulis normal subgroup theorem.

An example illustrating Theorem \ref{thm:lattice} starts
with an action on $S^1.$

\begin{ex}
Let $\G$ be the subgroup of $PSL(2,\Z[\sqrt{2}])$ generated by 
\[
A =
\begin{pmatrix}
\lambda^{-1} &  0\\
0 &  \lambda\\
\end{pmatrix}
\text{ and }
B =
\begin{pmatrix}
1 &  1\\
0 & 1 \\
\end{pmatrix}.
\]

where $\lambda = \sqrt{2} +1.$  Note $\lambda^{-1} = \sqrt{2} -1$.
\end{ex}

These matrices satisfy
\[
A^{-n}BA^n = 
\begin{pmatrix}
1 & \lambda^{2n} \\
0 & 1 \\
\end{pmatrix}
\]
and
\[
A^n BA^{-n} = 
\begin{pmatrix}
1 & \lambda^{-2n} \\
0 & 1 \\
\end{pmatrix}.
\]
It is easy to see that $m =\lambda^{2n} + \lambda^{-2n}$ is
an integer. Hence 
\[
(A^{-n}BA^n) (A^n BA^{-n}) =
\begin{pmatrix}
1 & \lambda^{2n} + \lambda^{-2n} \\
0 & 1 \\
\end{pmatrix}
= B^m.
\]
We have shown that $|B^m| \le 4n+2$ so 
\[
{
\liminf_{n \to \infty} \frac{|B^m|}{m} 
\le \liminf_{n \to \infty} \frac{4n+2}{\lambda^{2n}} = 0,
}
\]
so $B$ is distorted.  The group $\G$ acts naturally on $\RP^1$ (the lines 
through the origin in $\R^2$ ) which is diffeomorphic to $S^1$.  The
element $B$ has a single fixed point, the $x-$axis, and the only $B$
invariant measure is supported on this point.

In example 1.6.K of \cite{P} Polterovich considers the
embedding $\psi: \G \to PSL(2,\R) \times PSL(2,\R)$ where
$\psi(g) = (g, \bar g)$ with $\bar g$ denoting the conjugate of $g$
obtained by replacing an entry $a+b\sqrt{2}$ with $a-b\sqrt{2}.$
He points out that the image of $\psi$ is an irreducible
non-uniform lattice in a Lie group of real rank $2.$  Of course
$(B, \bar B) = (B, B)$ is a distortion element in $\psi(\G)$ and in 
the product action of $PSL(2,\R) \times PSL(2,\R)$ on $T^2 = S^1 \times S^1$ 
it has only one fixed point $(p, p)$ where $p$ is the fixed point of 
$B$ acting on $S^1.$  It is also clear that the only $(B, \bar B)$ invariant
measure is supported on this point.  It is easy to see that there are
elements of $\psi(\G)$ which do not fix this point, and hence there is
no probability measure invariant under all of $\psi(\G).$

Under the stronger hypothesis that the group $\G$ contains a
subgroup isomorphic to the Heisenberg group we can remove the
hypothesis that $\supp(\mu)$ is infinite and allow the case
that $S = S^2.$

\begin{thm}[\cite{FH3}]
Suppose $S$ is a closed oriented surface
with Borel probability measure $\mu$  and 
$\G$ is a finitely generated, almost simple group with
a subgroup isomorphic to the Heisenberg group.
Then any homomorphism
\[
\phi: \G \to \Diff_\mu(S)
\]
has finite image.
\end{thm}

\section{\color{\bl}Parallels between $\Diff(S^1)_0$ and $\Diff_\mu(S)_0$}
\noindent
In general there seem to be strong parallels between results about
$\Diff(S^1)_0$ and $\Diff_\mu(S)_0$.  For example, Witte's theorem
and our results above.   There are several other examples which we now
cite.

\begin{theorem}[H\"older]
Suppose $\G$ is a subgroup of $\Diff(S^1)_0$ which acts freely
(no non-trivial element has a fixed point).  Then $\G$ is abelian.
\end{theorem}

See $\cite{FS}$ for a proof.
There is an analog of this result for dimension two.  It is a
corollary of the following celebrated result.

\begin{theorem}[Arnold Conjecture: Conley-Zehnder]
Suppose $\omega$ is Lebesgue measure and
\[
f \in \Diff_\omega(\T^2)_0
\]
is in the commutator subgroup. Then $f$ has (at least three) fixed points.
\end{theorem}

\begin{cor}
Suppose $\G$ is a subgroup of $\Diff_\omega(\T^2)_0$ which acts freely.
Then $\G$ is Abelian.
\end{cor}
\begin{proof}
If $f$ is a commutator in $\G$.  Then by the theorem
of Conley and Zehnder it has a fixed point. Since $\G$ acts
freely only the identity element has fixed points.  If all
commutators of $\G$ are the identity then $\G$ is abelian.
\end{proof}


\begin{defn}
A group $\N$ is called {\em nilpotent} provided 
when we define
\[
\N_0 = \N,\  \N_i = [\N,\N_{i-1}],
\]
there is
an $n \ge 1$ such that $\N_n = \{e\}.$  Note if $n = 1$ it is
Abelian.
\end{defn}

\begin{theorem}[Plante - Thurston \cite{PT}]
\label{theorem:C2interval}
Let $N$ be a  nilpotent subgroup of $\Diff^2(S^1)_0.$ 
Then $N$ must be Abelian.
\end{theorem}

The result of Plante and Thurston requires the $C^2$ hypothesis as the 
following result shows.

\begin{theorem}[\cite{FF}]
Every finitely-generated, torsion-free nilpotent group is isomorphic
to a subgroup of $\Diff^1(S^1)_0$.
\end{theorem}

There is however an analogue of the Plante - Thurston Theorem for 
surface diffeomorphisms which preserve a measure.

\begin{theorem}[\cite{FH3}]
Let $\N$ be a 
nilpotent subgroup of $\Diff^1_\mu(S)_0$ with $\mu$ a probability
measure with $\supp(\mu) = S.$
If $S \ne S^2$ then $\N$ is Abelian, if $S = S^2$ then $\N$ is Abelian
or has an index 2 Abelian subgroup.
\end{theorem}

\noindent
\begin{proof} 
We sketch the proof in the case $genus(S) > 1$.
Suppose 
\[
\N = \N_1 \supset \dots \supset \N_m \supset \{1\}
\]
is the lower central series of $\N.$ 
then $\N_{m}$ is in the center of $\N.$  If $m > 1$ there is a 
non-trivial $f \in  \N_{m}$ and elements $g,h$ with $f = [g,h].$
No non-trivial element of $\Diff^1(S)_0$ has
finite order since $S$ has genus $> 1.$
So $g,h$ generate a Heisenberg group and $f$ is distorted.
Theorem~\ref{thm:distort} above says 
$\supp(\mu) \subset \Fix(f),$ but $\supp(\mu) = S$
so $f = id.$  This is a contradiction unless $m = 1$ and $\N$ is abelian.
\end{proof}

\vfill\eject
\section{Detecting Non-Distortion}

Given a diffeomorphism which we wish to prove is not distorted there are 
three properties, any one of which will give us the desired conclusion.
In this section we will define these properties and show they are sufficient
to establish non-distortion.  These properties are
\begin{itemize}
\item exponential growth of length of a curve
\item linear displacement in the universal cover
\item positive {\em spread}
\end{itemize}

\begin{defn}
If the surface $S$ is provided with a Riemannian metric
a smooth closed curve $\tau \subset S$ has a well defined length
$l_S(\tau)$.  Define the {\em exponential growth rate} by
\[
\egr(f,\tau) = \liminf_{n \to \infty}\frac{\log(l_S(f^n(\tau))}{n}.
\]
\end{defn}

This is easily seen to be independent of the choice of metric.

\begin{prop} \label{egr} If $G$ is a finitely generated subgroup of $\Diff(S)_0$ 
and $f \in G$ is distorted in $G$   then $\egr(f, \tau) = 0$ for all closed curves $\tau$.
\end{prop}

\begin{proof} Choose generators $g_1,\dots,g_j$ of $G$.  There exists  $C > 0$ such that $||Dg_i|| < C$ 
for all $i$.  Thus $l_S(g_i(\tau)) \le C l_S(\tau)$ for all $\tau$ and all $i$.  It follows that  
\[
\liminf_{n \to \infty}\frac{\log(l_S(f^n(\tau))}{n} \le \liminf_{n \to \infty}\frac{\log(l_S(\tau)) 
+ \log(C) |f^n|}{n} = 0.
\]
\end{proof}

\begin{defn}
Assume that $f \in \Homeo(S)_0$ and that $S \ne S^2$.   A metric $d$ on $S$ lifts to an equivariant metric $\ti d$ on the universal cover $\ti S$.  We say that $f$ has {\it linear displacement} if either of the following conditions hold.
\begin{enumerate}
\item $S \ne T^2$, $\ti f$ is the identity lift  and there exists $\ti x \in \ti S = H$ such that 
$$
\liminf_{n \to \infty} \frac{\ti d(\ti f^n(\ti x),\ti x)}{n} > 0.
$$
\item $S = T^2$ and there exist $\ti f$ and  $\ti x_1,\ti x_2 \in \ti S = \R^2$ such that
$$
\liminf_{n \to \infty} \frac{\ti d(\ti f^n(\ti x_1),\ti f^n(\ti x_2))}{n} > 0.
$$
\end{enumerate}
\end{defn}

\begin{prop} \label{no linear displacement} If $G$ is a finitely generated subgroup of $\Homeo(S)_0$ and $f \in G$ is distorted in $G$   then $f$ does not have linear displacement. 
\end{prop}

\begin{proof} We present only the case that $S$ has genus $> 1.$  For the full result
see \cite{FH3}.
In this case the identity lifts $\{\ti g: g \in G\}$ form a subgroup
$\ti G$ and $\ti f$ is a distortion element in $\ti G$.  Let $d$ be
the distance function of a Riemannian metric on $S$ and let $\ti d$ be
its lift to $H$.  For generators $g_1,\dots,g_j$ of $G$ there exists
$C > 0$ such that $\ti d(\ti g_i(\ti x),\ti x) < C$ for all $\ti x \in
H$ and all $i$.  It follows that
\[
\liminf_{n \to \infty} \frac{\ti d(\ti f^n(\ti x),\ti x)}{n} \le
\liminf_{n \to \infty} C \frac{|f^n|}{n} = 0.
\]
\end{proof}

The final ingredient we use to detect non-distortion is {\em spread} which 
we now define.  The following few paragraphs are taken almost verbatim from
\cite{FH3}.

Suppose that $f \in \Diff(S)_0$, that $\gamma \subset S$ is a smoothly
embedded path with distinct endpoints in $\Fix(f)$ and that $\beta$ is
a simple closed curve that crosses $\gamma$ exactly once. We want to
measure the relative rate at which points move \lq across $\gamma$ in
the direction of $\beta$\rq.

Let $A$ be the endpoint set of $\gamma$ and let $M$ be the surface
with boundary obtained from $S$ by blowing up both points of $A$. We
now think of $\gamma$ as a path in $M$ and of $\beta$ as a simple
closed curve in $M$.  Assume at first that $S \ne S^2$ and that $M$ is
equipped with a hyperbolic structure.  We denote the universal covering
space of $S$ by $H$ and the ideal points needed to compactify it by
$\sinfty.$  Choose non-disjoint extended lifts $\ti \beta \subset H \cup \sinfty$
and $\ti \gamma \subset H \cup \sinfty$ and let $T : H \cup \sinfty
\to H \cup \sinfty$ be the covering translation corresponding to $\ti
\beta$, i.e. $T^{\pm}$ are the endpoints of $\ti \beta$.  Denote
$T^i(\ti \gamma)$ by $\ti \gamma_i$. Each $\ti \gamma_i$ is an
embedded path in $H \cup \sinfty$ that intersects $\sinfty$ exactly in
its endpoints.  Moreover, $\ti \gamma_i$ separates $\ti \gamma_{i-1}$
from $\ti \gamma_{i+1}$.
 
An embedded smooth path $\alpha \subset S$ whose interior is disjoint
from $A$ can be thought of as a path in $M$. For each lift $\ti \alpha
\subset H \cup \sinfty$, there exist $a < b$ such that $\ti \alpha
\cap \ti \gamma_i \ne \emptyset$ if and only if $a < i < b$.  Define
$$
\ti L_{\ti \beta, \ti \gamma}(\ti \alpha) = \max\{0,b-a-2\}
$$
and  
$$
L_{\beta,\gamma}(\alpha) = \max\{\ti L_{\ti \beta,\ti \gamma}(\ti \alpha)\}
$$
 as $\ti \alpha$ varies over all lifts of $\alpha$.  

Suppose now that $S= S^2$ and hence that $M$ is the closed annulus.
In this case $\ti M$ is identified with $\R \times [0,1]$, $T(x,y) =
(x+1,y)$ and $\ti \gamma$ is an arc with endpoints in both components
of $\partial \ti M$.  With these modifications, $L_{\beta,\gamma}(\alpha)$ is
defined as in the $S \ne S^2$ case.

There is an equivalent definition of $L_{\beta,\gamma}(\alpha)$ that
does not involve covers or blowing up.  Namely,
$L_{\beta,\gamma}(\alpha)$ is the maximum value $k$ for which there
exist subarcs $\gamma_0 \subset \gamma$ and $\alpha_0 \subset \alpha$
such that $\gamma_0\alpha_0$ is a closed path that is freely homotopic
relative to $A$ to $\beta^k$. We allow the possibility that $\gamma$
and $\alpha$ share one or both endpoints.  The finiteness of
$L_{\beta,\gamma}(\alpha)$ follows from the smoothness of the arcs
$\alpha$ and $\gamma$.

\begin{defn}\label{defn:spread}
Define the {\em spread} of
$\alpha$ with respect to $f, \beta$ and $\gamma$ to be 
\[
\sigma_{f, \beta,\gamma}(\alpha) 
=  \liminf_{n \to \infty} \frac{L_{\beta,\gamma}( f^n \circ \alpha)}{n}.
\]
\end{defn}

Note that if $\gamma'$ is another smoothly embedded arc that crosses $\beta$
exactly once and that has the same endpoints as $\gamma$ then
$\sigma_{f, \beta,\gamma}(\alpha)= \sigma_{f, \beta,\gamma'}(\alpha)$
for all $\alpha$.  This follows from the fact that $\ti \gamma' $ is
contained in the region bounded by $\ti \gamma_j$ and $\ti
\gamma_{j+J}$ for some $j$ and $J$ and hence
$|L_{\beta,\gamma'}(\alpha) -L_{\beta,\gamma}(\alpha)| \le 2J$ for all
$\alpha$.

\begin{prop} \label{prop:spread}  
If $G$ is a finitely generated subgroup of $\Diff(S)_0$ and $f \in G$ is distorted in $G$   then $\sigma_{f,\beta,\gamma}(\alpha) = 0$ for all $\alpha,\beta,\gamma$.
\end{prop}

This proposition is proved via three lemmas which we now state.  For proofs see
\cite{FH3}.

\begin{lem} \label{lem: geodesic seg}  
Suppose that $g \in \Diff(S)$ and that 
$\eta$ and $\eta'$ are smoothly embedded geodesic arcs in
$S$ with length at most $D$.  There exists a constant $C(g)$, independent of $\eta$ and
$\eta'$ such that the absolute value of the
algebraic intersection number of any subsegment
of $g(\eta)$ with $\eta'$ is less than $C(g).$
\end{lem}

Let $\gamma$ be a fixed oriented geodesic arc in $S$ with length at
most $D$, let $A =\{x,y\}$ be its endpoint set and let $M$ be the
surface with boundary obtained from $S \setminus A$ by blowing up $x$
and $y$.  For each ordered pair $\{x',y'\}$ of distinct points in $S$
choose once and for all, an oriented geodesic arc $\eta = \eta(x',y')$
of length at most $D$ that connects $x'$ to $y'$ and choose $h_{\eta}
\in \Diff(S)_0$ such that $h_{\eta} (\gamma) = \eta,\ h_{\eta} (x) =
x',\ h_{\eta} (y) = y'.$ There is no obstruction to doing this since
both $\gamma$ and $\eta$ are contained in disks.  If $x = x'$ and
$y=y'$ we choose $\eta =\gamma$ and $h_{\eta} = id.$

Given $g \in \Diff(S)$ and an ordered pair $\{x',y'\}$ of distinct
points in $S$, let $\eta = \eta(x',y')$, $\eta' = \eta(g(x'),g(y'))$
and note that $g_{x',y'} := h_{\eta'}^{-1} \circ g \circ h_\eta$
pointwise fixes $A$. The following lemma asserts that although the
pairs $\{x',y'\}$ vary over a non-compact space, the elements of
$\{g_{x',y'}\}$ exhibit uniform behavior from the point of view of
spread.

\begin{lem} \label{lem: gamma-len}  With notation as above, the following hold for all $g \in \Diff(S)$.   
\begin{enumerate}
\item   There exists a constant $C(g)$  such that 
\[
L_{\beta,\gamma}(g_{x',y'}(\gamma)) \le C(g) \mbox{ for all } \beta \mbox{ and all } x',y'.   
\]
\item There exists a constant $K(g)$ such that 
\[
L_{\beta,\gamma}(g_{x',y'} (\alpha)) \le L_{\beta,\gamma}(\alpha) + K(g) \mbox{ for all } \beta, \mbox{  all } \alpha \mbox{ and all }  x',y'. 
 \]
\end{enumerate}
\end{lem}

\begin{lem} \label{T-len growth}  
Suppose that $g_i \in \Diff(S)_0,\  1 \le i\le k,$ that $f$  is in the group they generate and that   
  $|f^n|$ is the word length of $f^n$ in the generators $\{g_i\}$. Then there is a constant $C >0$ such that
$$
L_{\beta,\gamma}(f^n (\alpha)) \le L_{\beta,\gamma}(\alpha) + C |f^n| 
$$ 
for all $\alpha,\beta,\gamma$ and all $n >0.$
\end{lem}

\noindent{\bf Proof of Proposition~\ref{prop:spread}}
Since $f$ is distorted in $G$
\[
\liminf_{n \to \infty} \frac{|f^n|}{n} = 0.
\]
According to the definition of spread and 
Lemma~\ref{T-len growth}
we then have
\[
\sigma_{f,\beta,\gamma}(\alpha) 
=  \liminf_{n \to \infty} \frac{L_{\beta,\gamma}( f^n(\alpha))
}{n}
\le \liminf_{n \to \infty} \frac{L_{\beta,\gamma}(\alpha) + C |f^n|}{n}
= 0.
\]\qed

\section{Sketch of Theorem~\ref{thm:distort}}

The following proposition is implicit in the paper of Atkinson \cite{A}.
This proof is taken from \cite{FH1} but is essentially the same as an
argument in \cite{A}.

\begin{prop}\label{prop: atkinson}
Suppose $T: X \to X$ is an ergodic automorphism of a probability space
$(X,\nu)$ and let $\phi: X \to \R$ be an integrable function with
$\int \phi \ d\nu = 0.$ Let $S(n,x) = \sum_{i=0}^{n-1} \phi( T^i(x))$.
Then for any $\varepsilon >0$ the set of $x$ such that $|S(n,x)| <
\varepsilon$ for infinitely many $n$ is a full measure subset of $X$.
\end{prop}

\begin{proof}
Let $A$ denote the set of $x$ such that $|S(n,x)| < \varepsilon$ for
only finitely many $n$.  We will show the assumption $\mu(A) > 0$
leads to a contradiction.  Suppose $\mu(A) > 0$ and let $A_m$ denote
the subset of $A$ such that $|S(i,x)| < \varepsilon$ for $m$ or fewer
values of $i$.  Then $A = \cup A_m$ and there is an $N >0$ such that
$\mu(A_N) > p$ for some $p >0.$

The ergodic theorem applied to the characteristic function of $A_N$
implies that for almost all $x$ and all sufficiently large $n$ (depending
on $x$) we have 
\[
\frac{card( A_N \cap \{T^i(x)\ |\ 0 \le i < n\})}{n} > p.
\]

We now fix an $x \in A_N$ with this property.
Let $B_n = \{i\ |\ 0 \le i \le n 
\text{ and } T^i(x) \in A_N\}$ and $r = card(B_n)$; then $r > np$.
Any interval in $\R$ of length $\varepsilon$ which 
contains $S(i,x)$ for some $i \in B_n$  contains at most $N$ values of
$\{S(j,x) : j > i\}.$
Hence any interval of length $\varepsilon$ contains at most 
$N$ elements of $\{ S(i,x)\ |\ i\in B_n\}.$
Consequently an interval containing 
the $r$ numbers $\{ S(i,x)\ |\ i \in B_n\}$ must have length at least
$r\varepsilon/N$.  Since $r > np$ this length is $> np\varepsilon/N.$
Therefore
\[
\sup_{0 \le i \le n} |S(i,x)| > \frac{np\varepsilon}{2N},
\]
and hence by the ergodic theorem, for almost all $x \in A_N$ 
\[
\Big | \int \phi\ d\mu \Big |
= \lim_{n \to \infty} \frac{|S(n,x)|}{n} 
= \limsup_{n \to \infty} \frac{|S(n,x)|}{n}
> \frac{p\varepsilon}{2N} > 0.
\]
This contradicts the hypothesis so our result is proved.
\end{proof}

\begin{cor}\label{cor: atkinson}
Suppose $T: X \to X$ is an automorphism of a Borel probability space
$(X,\mu)$ and $\phi: X \to \R$ is an integrable function.
Let $S(n,x) = \sum_{i=0}^{n-1} \phi( T^i(x))$ and suppose 
$\mu(P) > 0$ where $P = \{x \ |\ \lim_{n \to \infty} S(n,x) = \infty\}.$
Let 
\[
\hat \phi(x) = \lim_{n \to \infty} \frac{S(n,x)}{n}.
\]
Then $\int_P \hat \phi \ d\mu > 0.$  In particular $\hat \phi(x) >0$ for a set
of positive $\mu$-measure.
\end{cor}
\begin{proof}
By the ergodic decomposition theorem there is a measure $m$ on the
space $\M$ of all $T$ invariant ergodic Borel measures on $X$ 
with the property that for any $\mu$ integrable function 
$\psi : X \to \R$ we have 
$\int \psi \ d\mu = \int_\M I(\psi,\nu) \ dm$ where $\nu \in \M$
and $I(\psi,\nu) = \int \psi\ d\nu.$

The set $P$ is $T$ invariant. Replacing $\phi(x)$ with 
$\phi(x) \X_P(x),$ where $\X_P$ is the characteristic function
of $P,$ we may assume that $\phi$ vanishes outside $P$.
Then clearly $\hat \phi(x) \ge 0$ for all $x$ for which it exists.
Let $\M_P$ denote $\{ \nu \in \M\ |\ \nu(P) > 0 \}$.  
If $\nu \in \M_P$ the fact that $\hat \phi(x) \ge 0$ and the ergodic theorem
imply that $I(\phi,\nu) = \int \phi\ d\nu  = \int \hat \phi\ d\nu  \ge 0$. 
Also Proposition~\ref{prop: atkinson}
implies that $\int \phi\ d\nu  = 0$ is impossible so
$I(\phi,\nu) > 0.$
Then $\mu(P) =  \int I( \X_P, \nu)\ dm = \int \nu(P)\ dm = 
\int_{\M_P} \nu(P)\ dm.$   This implies $m(\M_P) > 0$ since
$\mu(P) > 0.$

Hence
\[
\int \hat \phi \ d\mu = \int \phi \ d\mu = \int I( \phi, \nu)\ dm \ge \int_{\M_P} I( \phi, \nu)\ dm
>0
\]
since $I( \phi, \nu) >0$ for $\nu \in \M_P$ and $m(\M_P) > 0.$
\end{proof}

\noindent
{\bf Outline of the proof of Theorem~\ref{thm:distort}}

We must show that if $f\in \Diff_\mu(S)_0$ has infinite order and
$\mu(S\setminus \Fix(f)) > 0$ then $f$ is not distorted.  In light
of the results of the previous section this will follow from the
following proposition.

\begin{prop}
If $f\in \Diff_\mu(S)_0$ has infinite order and
$\mu(S\setminus \Fix(f)) > 0$ then 
one of the following holds:
\begin{enumerate}
\item There exists a closed curve $\tau$ such that 
$\egr(f, \tau) > 0.$
\item $f$ has linear displacement.
\item After replacing $f$ with some iterate $g=f^k$ and perhaps
passing to a two-fold covering $g:S \to S$ is isotopic to the identity
and there exist $\alpha, \beta, \gamma$ such that
the spread $\sigma_{f,\beta,\gamma}(\alpha) >0.$
\end{enumerate}
\end{prop}

The idea of the proof of this proposition is to first ask if 
$f$ is isotopic to the identity relative to $\Fix(f)$.  If not
there is a finite set $P \subset \Fix(f)$ such that $f$ is not
isotopic to the identity relative to $P$.  We then consider the
Thurston canonical form of $f$ relative to $P$.  If there is 
pseudo-Anosov component the property (1) holds.  If there are
no pseudo-Anosov components then there must be non-trivial
Dehn twists in the Thurston canonical form.  In this case it
can be shown that either (2) or (3) holds.  For details
see \cite{FH3}

We are left with the case that $f$ is isotopic to the identity
relative to $\Fix(f).$ There are several subcases.  It may be that $S$
has negative Euler characteristic and the identity lift $\ti f$ has a
point with non-zero rotation vector in which case (2) holds. It may be
that $S = T^2$ and there is a lift $\ti f$ with a fixed point and a
point with non-zero rotation vector in which case (2) again holds.

The remaining cases involve $M = S \setminus \Fix(f).$  A result
of Brown and Kister \cite{BK} implies that each component of $M$
is invariant under $f$.  If $M$ has a component which is an annulus
and which has positive measure then there is a positive measure set
in the universal cover of this component which goes to infinity in
one direction or the other.  In this case Corollary~\ref{cor: atkinson},
with $\phi$ the displacement by $\ti f$ in the covering space, 
implies there are points with non-zero rotation number.  Since points
on the boundary of the annulus have zero rotation number we can
conclude that (3) holds.

The remaining case is that there is a component of $M$ with positive
measure and negative Euler characteristic (we allow infinitely many
punctures).  In this case it can be shown that there is a simple closed
geodesic and a set of positive measure whose lift in the universal cover
of this component tends asymptotically to an end of the simple closed
geodesic.  An argument similar to the annular case then shows that
(3) holds.

More details can be found in \cite{FH3} including the fact that these
cases exhaust all possibilities.

\end{document}